\documentclass[11pt]{amsart}
\usepackage{latexsym,amssymb,amsmath}
\usepackage{enumerate}
\usepackage{amsfonts}
\usepackage{eucal}\usepackage{hyperref}
\usepackage{exscale}
\newcommand{\la}{\langle}
\newcommand{\ra}{\rangle}

\newcommand{\mc}{\mathcal}

\newcommand{\union}{\cup}

\newcommand{\inter}{\cap}
\newcommand{\Union}{\bigcup}

\newcommand{\Inter}{\bigcap}
\renewcommand{\to}{\rightarrow}

\newcommand{\restrict}{\upharpoonright}

\renewcommand{\and}{\,\&\,}

\newcommand{\NN}{{\mathbb{N}}}
\newcommand{\uhr}{\upharpoonright}

\newcommand{\ES}{\emptyset}

\newcommand{\sub}{\subseteq}
\newcommand{\wt}{\widetilde}
\newcommand{\CCC}{\mathcal{C}}

\newcommand{\ML}{Martin-L\"of}

\newcommand{\LLL}{\rm{Low}}
\newcommand{\DDD}{\mathcal{D}}
\newcommand{\DII}{\Delta^0_2}

\renewcommand{\MR}{\mbox{\rm \textsf{MR}}} 
\newcommand{\SR}{\mbox{\rm \textsf{SR}}}

\newcommand{\QQ}{\mathbb{Q}}

\renewcommand{\phi}{\varphi}

\newcommand{\RA}{\Rightarrow}

\newcommand{\LA}{\Leftarrow}

\newcommand{\bc}{\begin{center}}
\newcommand{\ec}{\end{center}}

\newcommand{\lra}{\longrightarrow}
\renewcommand{\NN}{\omega}

\newcommand{\recursive}{computable}
\newcommand{\computable}{computable}
\newcommand{\recursively}{com\-put\-ably} 
\newcommand{\computably}{com\-put\-ably}
\newcommand{\re}{r.e.} 
\newcommand{\ce}{r.e.}

\newcommand{\recursion}{computability}
\newcommand{\computablefrom}{computable from}

\newcommand{\CR}{\mbox{\rm \textsf{CR}}} 
\newtheorem{theorem}{Theorem}[section]

\newtheorem{definability lemma}[theorem]{Definability Lemma}

\newtheorem{proposition}[theorem]{Proposition}

\newtheorem{claim}[theorem]{Claim}

\newtheorem{definition}[theorem]{Definition}

\newtheorem{lemma}[theorem]{Lemma}
\newtheorem{corollary}[theorem]{Corollary}

\newtheorem{question}[theorem]{Question}

\title{Lowness for the Class of Schnorr Random Reals\thanks{Received by the editors October 15, 2004; accepted for 
publication (in revised form) July 19, 2005; published electronically January 6, 2006. The first author was supported by a Marie 
Curie Fellowship of the European Community Programme ``Improving Human Potential'' under contract HPMF-CT-2002-01888. 
The second author was partially supported by University of Auckland New Staff research grant 3603229/9343, and by the 
Marsden fund of New Zealand, 03-UOA-130. \URL sicomp/35-3/44632.html}}

\author{Bj\O rn Kjos-Hanssen}
\thanks{Department of Mathematics, University of Connecticut, Storrs, CT 06269 
(bjorn@math.uconn. edu).}

\author{Andr\'e Nies}
\thanks{Department of Computer Science, University of Auckland, Private Bag 92019, Auckland, New Zealand 
(andre@cs.auckland.ac.nz).} 
\author{Frank Stephan}
\thanks{School of Computing and Department of Mathematics, National 
University of Singapore, 3 Science Drive 2, Singapore 117543, Republic of Singapore (fstephan@comp.nus. edu.sg).}

\begin{document}
\begin{abstract}
We answer a question of Ambos-Spies and Ku\v{c}era in the
affirmative. They asked whether, when a real is low for Schnorr
randomness, it is already low for Schnorr tests.

\end{abstract}

\maketitle

\pagestyle{myheadings} 
\thispagestyle{plain} 
\markboth{B. KJOS-HANSSEN, A. NIES, AND F. STEPHAN}{SCHNORR RANDOM REALS}

\tableofcontents

\section{Introduction}
In an influential 1966 paper \cite{Martin-Lof:66}, \ML\ proposed an
algorithmic formalization of the intuitive notion of randomness for
infinite sequences of 0's and 1's. His formalization was based on an
effectivization of a test concept from statistics, by means of
uniformly recursively enumerable (\ce)~sequences of open sets. \ML's
proposal addressed some insufficiencies in an earlier algorithmic
concept of randomness proposed by Church \cite{Church:40}, who had
formalized a notion now called \computable\  stochasticity. However,
Schnorr \cite {Schnorr:71} criticized \ML's  notion as too strong,
because it was based on an \ce\  test concept rather than a
\computable\ notion of tests. He suggested that one should base a
formalization of randomness on \computable\ betting strategies (also
called martingales), in a way that would still overcome the problem
that Church's concept was too weak. In present terminology, a real
$Z$ is \computably\ random if no computable betting strategy
succeeds  along $Z$; that is, for each computable betting strategy
there is a finite upper bound on the capital that it reaches. The
real $Z$ is Schnorr random if no martingale succeeds
\emph{effectively}. Here effective
 success means that  the capital at $Z\uhr n$ exceeds $f(n)$ infinitely often, 
for some unbounded \recursive\  function $f$. See \cite{Ambos-Spies.Kucera:00} 
for more on the history of these ideas.

We recall some definitions. The Cantor space $2^\omega$ is the set
of infinite binary sequences; these are called \emph{reals} and are
identified with a set of integers, i.e., subsets of $\omega$. If
$\sigma\in 2^{<\omega}$, that is, $\sigma$ is a finite binary
sequence, then we denote by $[\sigma]$ the set of reals that extend
$\sigma$. These form a basis of clopen sets for the usual discrete
topology on $2^\omega$. Write $|\sigma|$ for the length of
$\sigma\in 2^{<\omega}$. The Lebesgue measure $\mu$ on $2^\omega$ is
defined by stipulating that $\mu[\sigma]=2^{-|\sigma|}$. With every
set $U\subseteq 2^{<\omega}$ we associate the open set
$[U]^\preceq=\bigcup_{\sigma\in U}[\sigma]$. The empty sequence is
denoted $\lambda$. If $\sigma, \tau\in 2^{<\omega}$ and $\sigma$ is
a prefix of $\tau$, then we write $\sigma\preceq\tau$. If $\sigma\in
2^{<\omega}$ and $i\in\{0,1\}$, then $\sigma i$ denotes the string of
length $|\sigma|+1$ extending
$\sigma$ whose final entry is $i$. The concatenation of two 
strings $\sigma$ and $\tau$ is denoted $\sigma\tau$. The empty set 
is denoted $\emptyset$, and inclusion of sets is denoted by $\subseteq$. 
If $A$ is a real and $n\in\omega$, then $A\uhr n$ is the prefix of $A$ 
consisting of the first $n$ bits of $A$. Letting $A(n)$ denote bit $n$ of $A$, 
we have $A\uhr n=A(0)A(1)\cdots A(n-1)$.   

Given $\alpha\in 2^{<\omega}$ and a measurable set $C\subseteq
2^\omega$, we let  $\mu_\alpha C=\frac{\mu (C\inter [\alpha])}{\mu
[\alpha]}$. For an open set $W$ we let
$$W|\sigma=\bigcup\left\{[\tau]: \tau\in 2^{<\omega}, [\sigma\tau]\subseteq W\right\}.$$
Note in particular that $\mu_\sigma W=\mu (W|\sigma)$ and
$\mu_\lambda W=\mu W$.

Fixing some effective correspondence between the set of finite
subsets of $\omega$ and $\omega$, we let $D_e$ be the $e$th finite
subset of $\omega$ under this correspondence. In other words, $e$ is
a strong, or canonical, index for the finite set $D_e$. Similarly,
we let $S_e$ be the $e$th finite subset of $2^{<\omega}$ under a
suitable correspondence. Thus $S_e$ is a finite set of strings, and
$[S_e]^\preceq=\union_{\sigma\in S_e}[\sigma]$ is then the clopen
set coded by $e\in\omega$. We use the Cantor pairing function, namely the
bijection $p:\omega^2\to\omega$ given by
$p(n,s)=\frac{(n+s)^2+3n+s}{2}$, and write $\la n,s\ra=p(n,s)$.

A {\it Martin-L\"of test}  is a set $U\subseteq\omega\times
2^\omega$ such that $\mu U_n\le 2^{-n}$, where $U_n$ denotes the
$n$th section of $U$, and $U_n$ is a $\Sigma^0_1$ class, uniformly
in $n$. If, in addition, $\mu U_n$ is a \computable\ real, uniformly
in $n$, then $U$ is called a {\it Schnorr test}. $Z$ is {\it
Martin-L\"of random} if for each Martin-L\"of test $U$ there is an
$n$ such that $Z\not\in U_n$, and {\it Schnorr  random} if for each
Schnorr test $U$ there is an $n$ such that $Z\not\in U_n$. The
notion of Schnorr randomness is unchanged if we instead define a
Schnorr test to be a Martin-L\"of test for which $\mu U_n=2^{-n}$
for each $n\in\omega$.

Concepts encountered in \recursion\ theory are usually based on some
notion  of computation, and therefore have relativized forms. For
instance, we may relativize the tests and randomness notions above
to an oracle $A$. If $\mathcal{C}=\{X: X$ is Martin-L\"of random$\}$,
then the relativization is $\mathcal{C}^A=\{X: X$ is Martin-L\"of
random relative to $A \}$ (meaning that $\Sigma^0_1$ classes are
replaced by $\Sigma^{0,A}_1$ classes). In general, if $\mathcal{C}$
is such a relativizable class, we say that $A$ is {\it low for
$\mathcal{C}$} if $\CCC^A =\CCC$. If $\mathcal{C}$ is a randomness
notion, more computational power means a smaller class, namely
$\CCC^A \subseteq \CCC$ for any $A$. Being low for $\mathcal{C}$
means having small computational power (in a sense that depends on
$\mathcal{C}$). In particular, the low-for-$\mathcal{C}$ reals are
closed downward under Turing reducibility.

The randomness notions for which lowness was first considered are
Martin-L\"of and Schnorr randomness. Ku\v{c}era  and Terwijn
\cite{Kucera.Terwijn:99} constructed a noncomputable\ \ce\  set of
integers $A$ which is low for Martin-L\"of randomness,
answering a question of Zambella \cite{Zambella:90}. In the paper 
\cite{Terwijn.Zambella:01} it is shown that there are continuously 
many reals that are low for Schnorr randomness.

An important difference between the two randomness notions is that
for Martin-L\"of randomness, but not for Schnorr randomness, there
is a universal test $R$. Thus, $Z$ is not Martin-L\"of random if{}f
$Z \in \bigcap_{b \in \NN} R_b$. Therefore, in the Schnorr case, an
apparently stronger lowness notion is being {\it low for Schnorr
tests}, or \emph{$S_0$-low} in the terminology of
\cite{Ambos-Spies.Kucera:00}: $A$ is low for Schnorr tests if for
each Schnorr test $U^A$ relative to $A$ there is an unrelativized
Schnorr test $V$ such that $\bigcap_n U^A_n\subseteq \bigcap_n V_n$.
This implies that $A$ is low for Schnorr randomness, or
\emph{$S$-low} in the terminology of \cite{Ambos-Spies.Kucera:00}.
Ambos-Spies and Ku\v{c}era asked if the two notions coincide. We answer 
this question in the affirmative.

Terwijn and Zambella \cite{Terwijn.Zambella:01} actually constructed 
oracles $A$ that are low for Schnorr tests.
They first gave a characterization of this lowness property via a
notion of traceability, a restriction on the possible sequence of
values of the functions \computablefrom\ $A$. They showed that $A$
is low for Schnorr tests if{}f $A$ is \computably\  traceable (see
formal definition in the next section). Then they constructed
continuously many \computably\  traceable reals. We answer the question
of Ambos-Spies and Ku\v{c}era by showing that each  real which is
low for Schnorr randomness is in fact \computably\  traceable.

Towards this end, it turns out to be helpful to have a more general
view of lowness. We consider lowness for any pair of randomness
notions $\mathcal{C}$, $\mathcal{D}$ with $\mathcal{C} \sub
\mathcal{D}$.

\begin{definition}
$A$ is in $\LLL(\mathcal{C},\mathcal{D})$ if $\mathcal{C} \subseteq
\mathcal{D}^A$. We write $\LLL(\mathcal{C})$ for
$\LLL(\mathcal{C},\mathcal{C})$.
\end{definition}

Clearly, if $\mathcal{C} \sub  \wt \CCC \sub \wt \DDD \sub
\mathcal{D}$ are randomness notions, and the inclusions relativize
(so $\wt\DDD^A\subseteq\mathcal{D}^A$ for each real $A$), then
$\LLL(\wt \CCC,\wt \DDD) \sub \LLL(\mathcal{C},\mathcal{D})$. That
is, we make the class $\LLL(\wt \CCC,\wt \DDD)$  larger by
decreasing $\CCC$ or increasing $\DDD$. Let $\MR, \CR$, and $\SR$
denote the classes of Martin-L\"of random, \computably\  random
(defined below), and Schnorr random reals, respectively. Thus, for
instance, $\LLL(\MR, \CR)$  is the class of oracles $A$ such that
each Martin-L\"of random real is \computably\  random in $A$. We
will characterize lowness for any pair of randomness notions
$\mathcal{C} \sub \mathcal{D}$ with $\mc C,\mc D\in
\{\MR,\CR,\SR\}$.

Recall that $\Omega$ denotes the halting probability of a universal
prefix machine. $\Omega$  is a Martin-L\"of random
\emph{r.e.~real}, i.e., a real that can be
effectively approximated from below. Given $\mc D\supseteq\MR$, an
interesting lowness notion obtained by weakening $\LLL(\MR, \DDD)$
is $\LLL(\{\Omega\},\DDD)$. That is, instead of $\MR \subseteq
\DDD^A$ one merely requires that $\Omega \in \DDD^A$. We denote this
class by $\LLL(\Omega, \DDD)$. In \cite{Nies.Stephan.ea:XX}, the
case $\DDD =\MR$ is studied. The authors show that the class coincides with
$\LLL(\MR)$ on the $\DII$ reals but not in general. In fact, a
Martin-L\"of random real is 2-random if{}f it is in $\LLL(\Omega,
\MR)$.

Here we investigate the class $\LLL(\Omega, \SR)$. We show that $A$
is Low(\MR, \SR) if{}f $A$ is \ce\  traceable. Moreover, the weaker
assumption $\Omega \in \SR^A$ still implies  that $A$ is array
\recursive\  (there is a function $f \le_{wtt} \ES'$ bounding all
functions \computablefrom\ $A$, on almost all inputs). Thus for \ce\
sets of integers $A$, $A$ being Low(\MR, \SR) is in fact equivalent
to $\Omega \in$ \SR$^A$ by Ishmukhametov \cite{Ishmukhametov:99}. We
also provide an example of a real $A$ which is array \recursive\
but not $\LLL(\Omega, \SR)$.

\section{Main concepts}\label{forward}

\subsection{Martingales}  For our purposes, a {\it martingale}  is a  function
$M: 2^{< \omega} \mapsto \QQ $ (where $\QQ$ is the set of rational
numbers) such that (i) the domain of $M$ is $2^{< \omega} $, or
$2^{\leq n}=\{\sigma\in 2^{<\omega}:|\sigma|\le n\}$ for some $n$,
(ii) $M(\lambda ) \leq 1$, and (iii) $M$ has the martingale property
$M(x0)+M(x1)= 2M(x)$ whenever the strings $x0,x1$ belong to the
domain of $M$. A martingale $M$ \emph{succeeds} on a sequence $Z\in
2^\omega$ if
 \[\limsup_{n\rightarrow\infty} M(Z\uhr n)=\infty.\]
\noindent A real is {\it \computably\  random} if no \computable\
martingale succeeds on it.

A martingale $M$ {\it effectively succeeds} on a sequence $Z$ if
there is a nondecreasing and unbounded \computable\  function
$h:\omega\lra\omega$ such that
\[\limsup_{n\rightarrow\infty} M(Z\uhr n)-h(n) >0.\]

\noindent Equivalently (since we are considering integer-valued
functions), $\exists^\infty n$ $M(Z\uhr n)>h(n)$.
 We can now state the characterization of Schnorr randomness
in terms of martingales: a real $Z$ is Schnorr random iff
no \computable\  martingale effectively succeeds on $Z$.

\subsection{Traceability}
Let $W_e$ denote the $e$th \ce\  set of integers in some standard
list. A real $A$ is {\it \ce\  traceable\/} if there is a
\computable\  function $p$, called a \emph{bound}, such that for
every $f\le_T A$ there is a \computable\  function $r$ such that
for all $x$ we have $|W_{r(x)}|\le p(x)$ and $f(x)\in W_{r(x)}$.

The following is a stronger notion than \ce\ traceability. $A$ is
{\it \computably\  traceable} if there is a \recursive\  $p$ such
that for every $f\le_T A$ there is a \recursive\  $r$ such that for
all $x$ we have $|D_{r(x)}|\le p(x)$ and $f(x)\in D_{r(x)}$.

It is interesting to notice  that it does not matter what bound $p$
one chooses as a witness for  traceability; see the following.

\begin{proposition}[see Terwijn and Zambella \cite{Terwijn.Zambella:01}] \label{tracebound}
Let $A$ be a real that is \computably\  traceable with bound $p$.
Then for any monotone and unbounded \computable\  function $p'$, $A$
is \computably\  traceable with bound $p'$. The same holds for \ce\
traceability.
\end{proposition}

The result of Terwijn and Zambella is the  following.

\begin{theorem}[see \cite{Terwijn.Zambella:01}] \label{TZbigtheorem} A real $A$ 
is low for Schnorr tests if{}f $A$ is \computably\  traceable. 
\end{theorem}

\section{Statement of the main result}\mbox{}

\begin{theorem}  \label{allonlow} \hfill \mbox{}
\begin{enumerate}
\item[{\rm (I)}]  $A$ is $\LLL(\MR, \SR)$ if{}f $A$ is \ce\ traceable.
\item[{\rm (II)}]  $A$ is $\LLL(\CR, \SR)$  if{}f $A$ is $\LLL(\SR)$
if{}f $A$ is \computably\  traceable.
\end{enumerate}
\end{theorem}

  We make some remarks about the proofs and fill in the details in the next section.
  We obtain Theorem \ref{allonlow}(I) by modifying the methods in \cite{Terwijn.Zambella:01} to the case of
\ce\  traces instead of \computable\  ones.

As for Theorem \ref{allonlow}(II), by Theorem \ref{TZbigtheorem} if
$A$ is computably traceable, then $A$ is low for Schnorr tests. Hence
$A$ is certainly $\LLL(\SR)$, and therefore also $\LLL(\CR,\SR)$. It
remains only to show that each real $A$ $\in$ $\LLL(\CR, \SR)$ is
\computably\ traceable. To see that this is so, take the following
three steps:

1. Recall that $A$ is {\it  hyperimmune-free} if for each $g\le_T
A$ there is a \recursive\  $f$ such that for all $x$ we have
$g(x)\le f(x)$. As a first step towards proving Theorem
\ref{allonlow}(II), Bedregal and Nies \cite{Bedregal.Nies:03} showed
that each $A \in \LLL(\CR, \SR)$ is  hyperimmune-free (see Lemma
\ref{andre} below).  To see this, assume that $A$ is not, so there
is a function $g \le_T A$ not dominated by any \computable\
function $f$. Define a martingale $L\le_T A$ which  succeeds in the
sense of Schnorr, with the \computable\  lower bound $h(n) = n/4$,
on some $Z \in \CR$. One uses here that $g$ is infinitely often
above the running time of each \computable\  martingale. (Special
care has to be taken with the partial martingales, which results in
a real $Z$ that is only $\Delta^0_3$.)

2. If $A$ is  hyperimmune-free   and \ce\ traceable, then $A$ is
\computably\  traceable. If we let $g\le_T A$, then the first stage
where $g(x)$ appears in a given trace for $g$ can be computed
relative to $A$.

3. Now each $A$ in $\LLL(\CR, \SR)$ is \ce\  traceable by Theorem
\ref{allonlow}(I), and hence by the above is \computably\  traceable, and
Theorem \ref{allonlow}(II) follows.

We discuss lowness for the remaining pairs of randomness notions.
 Nies has shown that $A$ is $\LLL(\MR, \CR)$ if{}f $A$ is $\LLL(\MR)$ if{}f $A$ is $K$-trivial, where $A$ is $K$-trivial
if for all $n \ K(X\uhr n) \le K(n)+O(1)$ (see \cite{Nies:AM}). Here
$K(\sigma)$ denotes the prefix-free Kolmogorov complexity of
$\sigma\in 2^{<\omega}$. Finally, he shows that  a real $A$ which is
$\LLL(\CR)$ is computable; namely, $A$ is both $K$-trivial and
hyperimmune-free. Since all $K$-trivial reals are $\Delta^0_2$,
 and all hyperimmune-free $\Delta^0_2$ reals are \computable, the conclusion follows.

\section{Proof of the main result}

We first need to develop a few useful facts from measure theory.

\begin{definition}
A measurable set $A$ has density $d$ at a real $X$ if
\end{definition}
$$\lim_{n\to\infty}\mu_{(X\restrict n)}A=d.$$

A basic result is the following.

\begin{theorem}[Lebesgue density theorem]
Let $\Xi(A)=\{X: A$ has density $1$ at $X \}$. If $A$ is a measurable
set, then so is $\Xi(A)$, and the measure of the symmetric difference
of $A$ and $\Xi(A)$ is zero.
\end{theorem}

\begin{corollary}\label{LebesgueDensity}
Let $C$ be a measurable subset of $2^\omega$, with $\mu C>0$. Then
for each $\delta<1$ there is an $\alpha\in 2^{<\omega}$ such that
$\mu_\alpha C\ge\delta$.
\end{corollary}

We will use the following consequence of Corollary
\ref{LebesgueDensity}.

\begin{lemma}\label{weakening}
Let $0<\epsilon\le 1$. If $U_n$, $n\in\omega$, and $V$ are open
subsets of $2^\omega$ with $\Inter_{n\in\omega} U_n\subseteq V$ and
$\mu V<\epsilon$, then there exist $\sigma$ and $n$ such that
$\mu_\sigma(U_n-V)=0$ and $\mu_\sigma V<\epsilon$.
\end{lemma}
\unskip

\begin{proof} Suppose otherwise; we shall obtain a
contradiction by constructing a real in $\Inter_{n\in\omega} U_n-V$.
Let $\sigma_0=\lambda$ and assume we have defined $\sigma_n$ such
that $\mu_{\sigma_n} V<\epsilon$. By hypothesis,
$\mu_{\sigma_n}(U_n-V)>0$, and thus there is a $[\tau]\subseteq U_n$ such
that $\mu_{\sigma_n}([\tau]-V)>0$. In particular,
$\tau\succeq\sigma_n$ and $\mu_\tau V<1$. Let $C=2^\omega-V$, a
closed and hence measurable set. By Corollary \ref{LebesgueDensity}
applied to $C$ (and with $2^\omega$ replaced by $[\tau]$), there
exists $\sigma_{n+1}\succeq\tau$ such that
$\mu_{\sigma_{n+1}}V<\epsilon$. Let $X$ be the real that extends all
$\sigma_n$'s constructed in this way. Since $[\sigma_{n+1}]\subseteq
U_n$ for all $n$, we have that $X\in\Inter_{n\in\omega} U_n$. However, 
$[\sigma_n]\not\subseteq V$ for every $n$, so, since $V$ is open,
$X\not\in V$. This contradiction completes the proof.\qquad
\end{proof}

We now get to the proof of Theorem \ref{allonlow}.
First we show Theorem \ref{allonlow}(I), namely, that $A$ is Low(\MR, \SR) if{}f $A$ is \ce\ traceable. 
We start with the ``$\LA$'' direction.

\begin{lemma} If $A$ is \ce\ traceable, then $A$ is $\LLL(\MR,\SR)$.
\end{lemma}
\unskip

\begin{proof}Assume that $A$ is \ce\ traceable and that $U^A$
is a Schnorr test relative to $A$. Let $U^A_{n,s}$, $n,s\in\omega$,
be clopen sets, $U^A_{n,s}\subseteq U^A_{n,s+1}$,
$U^A_n=\bigcup_{s\in\omega} U^A_{n,s}$, such that the $U^A_{n,s}$
are $\Delta^{0,A}_1$ classes uniformly in $n$ and $s$. As $\mu
U^A_n=2^{-n}$, we may assume that $\mu U^A_{n,s}> 2^{-n}(1-2^{-s})$.
Let $f$ be an $A$-computable function such that $[S_{f(\la
n,s\ra)}]^\preceq=U^A_{n,s}$. Since $A$ is \ce\ traceable and
$f\le_T A$, we can let $T$ be an \ce\ trace of $f$. By Proposition
\ref{tracebound}, we may choose $T$ such that in addition $|T_x|\le
x$ for each $x>0$.

We now want to define a subtrace $\hat T$ of $T$, i.e., $\hat
T_{\la n,s\ra}\subseteq T_{\la n,s\ra}$ for each $n,s$. The intent
is that the open sets defined via $\hat T$ are small enough to give
us a Martin-L\"of test containing $\inter_{n\in\omega} U^A_n$, and
nothing important is in $T_{\la n,s\ra}-\hat T_{\la n,s\ra}$. Thus let
$\hat T_{\la n,s\ra}$ be the set of $e\in T_{\la n,s\ra}$ such that
$2^{-n}(1-2^{-s})\le\mu [S_e]^\preceq\le 2^{-n}$ and
$[S_e]^\preceq\supseteq [S_d]^\preceq$ for some $d\in\hat T_{\la
n,s-1\ra}$, where $\hat T_{\la n,-1\ra}=\omega$. Let
$$V_n=\Union\left\{[S_e]^\preceq: e\in \hat T_{\la n,s\ra},\,\, s\in\omega\right\}.$$

\noindent Then $\mu V_n\le 2^{-n} |\hat T_{\la
n,0\ra}|+\sum_{s\in\omega} 2^{-s} 2^{-n} |\hat T_{\la n,s\ra}|$.
Since $|\hat T_{\la n,s\ra}|\le |T_{\la n,s\ra}|\le \la n,s\ra$ for
$\la n,s\ra\ne 0$, and $\la n,s\ra$ has only polynomial growth in
$n$ and $s$, it is clear that $\sum_{s\in\omega}2^{-s}2^{-n}|\hat
T_{\la n,s\ra}|$ is finite and goes effectively to $0$ as
$n\to\infty$; hence the same can be said of $\mu V_n$. Thus there
is a recursive function $f$ such that $\mu V_{f(n)}\le 2^{-n}$. Let
$\tilde V_n=V_{f(n)}$. Then $\tilde V$ is a Martin-L\"of test and
$\Inter_n U^A_n\subseteq \Inter_n \tilde V_n$. That is, each Schnorr
test relative to $A$ is contained in a Martin-L\"of test. It follows
that each real that is Martin-L\"of random is Schnorr random
relative to $A$, and the proof is complete.\qquad\end{proof}

Next we will show the ``$\RA$'' direction of Theorem
\ref{allonlow}(I). The proof is similar to the ``$\RA$'' of Theorem
\ref{TZbigtheorem}.

\begin{definition} For $k,l\in\omega$ define the clopen set
\[B_{k,l}=\bigcup\left\{[\tau1^k]:\tau\in 2^{<\omega}, |\tau|=l\right\},\]
where $1^k$ is a string of $1$'s of length $k$.
\end{definition}

Note that $\mu B_{k,l}=2^{-k}$ for all $l$.

\begin{lemma}
If $A\in 2^\omega$ is $\LLL(\MR,\SR)$, then $A$ is \ce\ traceable.
\end{lemma}
\unskip

\begin{proof} Note that  $A$ is Low(\MR, \SR) if{}f for
every Schnorr test $U^A$ relative to $A$, $\bigcap_{n \in \NN} U^A_n
\sub \bigcap_{b \in \NN} R_b$ (recall that $R$ is a universal
Martin-L\"of test).

Oversimplifying a bit, one can say that the proof below goes as
follows. We code a given $g\le_T A$ into a Schnorr test $U^g$
relative to $A$. Then, by hypothesis, $\bigcap_n
U^g_n\subseteq\bigcap_n R_n$; in fact we will use only the fact that
$\bigcap_n U^g_n\subseteq R_3$. We then define an \ce 
trace $T$; namely, $T_k$ is the set of $l$ such that $B_{k,l}-R_3$ has small
measure in some sense. Since $R_3$ has rather small measure,
$B_{k,l}-R_3$ will tend to have big measure, which means that there
will be only a few $l$ for which $B_{k,l}-R_3$ has small measure; in
other words, $T_k$ has small size. Moreover, we make sure $T$ is a
trace for $g$ so that $A$ is \ce\ traceable.

We now give the proof details. Suppose that we want to find a trace for a
given function $g\le_T A$. We define the test $U^g$ by stipulating
that
\[ U^g_n=\Union_{k>n}B_{k,g(k)}.\]
It is easy to see that $\mu U^g_n$ can be approximated \computably\
in $A$, so after taking a subsequence of $U^g_n$, $n\in\omega$, we
may assume that $U^g$ is a Schnorr test relative to $A$. Hence, by
assumption,  $\Inter U^g \sub \bigcap_{b \in \NN} R_b$. Thus  $V=R_3$
contains $\Inter U^g$ and $\mu V<\frac{1}{4}$. We may assume
throughout that $g(k)\ge k$ for every $k$ because from a trace for
$g(k)+k$ one can obtain  a trace for $g$ with the same bound. By
Lemma \ref{weakening}, there exist $\sigma$ and $n$ such that
$\mu_\sigma(U^g_n-V)=0$ and $\mu_\sigma V<1/4$. As $U^g_0\supseteq
U^g_1\supseteq\cdots$, we can choose $\sigma$ and $n$ with the
additional property $n\ge|\sigma|$. Hence for each $k>n$, we have
$g(k)\ge k>n\ge |\sigma|$ and hence $g(k)\ge |\sigma|$.

Let $\tilde V=V|\sigma$, let $\tilde g(k)=\max\{0,g(k)-|\sigma|\}$,
and take
\[ T_k=\left\{l:\mu(B_{k,l}-\tilde V)<2^{-(l+3)}\right\}.\]

\noindent Note that for each $l\in\omega$, if $l\ge |\sigma|$, then
$B_{k,l}|\sigma=B_{k,l-|\sigma|}$. Thus, since $g(k)\ge |\sigma|$,
$$U^g_n|\sigma=\Union_{k>n} B_{k,g(k)}|\sigma = \Union_{k>n} B_{k,g(k)-|\sigma|} = U^{\tilde g}_n,$$
and so $\mu(U^{\tilde g}_n-\tilde V)=\mu_\sigma(U^g_n-V)=0$.
Hence $\tilde g(k)\in T_k$ for  all $k>n$.

Since $\tilde V$ is a $\Sigma^0_1$ class, it is evident that $T$ is
an \re\ set of integers; indeed $B_{k,l}-\tilde V$ is a $\Pi^0_1$
class, and thus we can enumerate the fact that certain basic open sets
$[\sigma]$ are disjoint from it, until the measure remaining is as
small as required. A trace for $g$ is obtained as follows:
\begin{equation*} G_k=
\begin{cases}
   \{l+|\sigma|:l\in T_k\}  & \text{if } k>n,
\\ \{g(k)\}                 & \text{if } k\le n.
\end{cases}
\end{equation*}

We now show that $G$ is a trace for $g$; i.e., for all $k\in\omega$,
$g(k)\in G_k$. If $k\le n$, then this holds by definition of $G_k$;
thus suppose $k>n$. Then $g(k)>k>n>|\sigma|$, so $\tilde
g(k)=g(k)-|\sigma|$, so $g(k)=\tilde g(k)+|\sigma|$. As $k>n$,
$\tilde g(k)\in T_k$ and hence $g(k)\in G_k$.

Clearly $G$ is \re; thus it remains to show that $|G_k|$ is
\recursively\ bounded, independently of $g$. As $|G_k|=|T_k|$ for
$k>n$ and $|G_k|=1$ for $k\le n$, this is a consequence of Lemma
\ref{lalala} below.\qquad\end{proof}

\begin{lemma}\label{lalala}
If $\tilde V$ is a measurable set with $\mu\tilde V<\frac{1}{4}$, and $T_k=\{l:\mu(B_{k,l}-\tilde V)<2^{-(l+3)}\}$, 
then for $k\ge 1$, $|T_k|<2^kk$. 
\end{lemma}
\unskip

\begin{proof} Observe that, by definition of $T_k$,
$$\sum_{l\in T_k}\mu(B_{k,l}-\tilde V)<\sum_{l\in T_k}2^{-(l+3)}\le\frac{1}{8}\sum_{l\in\omega}2^{-l}=\frac{1}{4},$$
\noindent so
$$\mu\Union_{l\in T_k} B_{k,l}-\mu\tilde V\le\mu\Union_{l\in
T_k}(B_{k,l}-\tilde V)\le\frac{1}{4}.$$
\noindent As $\mu \tilde V<\frac{1}{4}$, we obtain that
$$\mu\Union_{l\in T_k} B_{k,l}<\frac{1}{2}.$$

As observed above, $\mu B_{k,l}=2^{-k}$. Moreover, for $k$ fixed, the
$B_{k,l}$'s are mutually independent as soon as the $l$'s are taken
sufficiently far apart. In fact, sufficiently far here means a
distance of $k$. So for $k\ge 1$ we let $T^*_k$ be a subset of
$T_k$ consisting of $\big\lfloor\frac{|T_k|}{k}\big\rfloor$
elements, all of which are sufficiently far apart. (Here $\lfloor
a\rfloor$ is the greatest integer $\le a$.) To show that such a set
exists we may assume we are in the worst case, where the elements of
$T_k$ are closest together: say, $T_k=\{0,\ldots,|T_k|-1\}$. Then let
$T^*_k=\{mk:0\le m\le \big\lfloor\frac{|T_k|}{k}\big\rfloor-1\}$.
As $\big(\big\lfloor\frac{|T_k|}{k}\big\rfloor-1\big)k\le
|T_k|-k\le |T_k|-1\in T_k$, this makes $T^*_k\subseteq T_k$. Write
$\alpha=\big\lfloor\frac{|T_k|}{k}\big\rfloor$. We now have
$$\mu\Inter_{l\in T_k} (2^\omega-B_{k,l})\le\mu\Inter_{l\in T^*_k}
(2^\omega-B_{k,l})=(1-2^{-k})^{\alpha}$$
\noindent and hence
$$1-(1-2^{-k})^{\alpha}\le 1-\mu \Inter_{l\in
T_k}(2^\omega-B_{k,l})=\mu 2^\omega-\mu \Inter_{l\in
T_k}(2^\omega-B_{k,l})$$
$$\le \mu\left(2^\omega-\Inter_{l\in
T_k}(2^\omega-B_{k,l})\right)=\mu\Union_{l\in T_k}
B_{k,l}<\frac{1}{2}.$$

\noindent From the inequality above we obtain, letting $m=2^k-1$,
$$\left(1-\frac{1}{m+1}\right)^{\alpha}=(1-2^{-k})^{\alpha}> \frac{1}{2}$$
or $\left(\frac{m+1}{m}\right)^{\alpha}<2$. Now suppose $\alpha\ge
m$. Then
$\left(\frac{m+1}{m}\right)^{\alpha}\ge\left(\frac{m+1}{m}\right)^m\ge
2$ as $(m+1)^m\ge m^m+m^{m-1}{m\choose 1}=2m^m$. Thus we conclude
$\alpha<m=2^k-1$. Now, by the definition of $\alpha$, we have
$\frac{T_k}{k}\le\alpha+1<2^k$ and so $|T_k|<2^k k$; this completes
the proof.\qquad
\end{proof}

In order to prove Theorem \ref{allonlow}(II), recall that, by
Theorem~\ref{TZbigtheorem}, each \computably\  traceable real is
$\LLL(\SR)$. Thus it suffices to show that each $\LLL(\CR, \SR)$
real is \computably\  traceable. The first ingredient for showing this
is the following result from \cite{Bedregal.Nies:03}.

\begin{lemma}\label{andre} 
 If $A$ is $\LLL(\CR,\SR)$, then $A$ is hyperimmune-free.
\end{lemma}

{\it Proof.}  Suppose $A$ is not hyperimmune-free, so that
there is a function $g\leq_T A$ not  dominated by any \computable\
function. Thus, for each \computable\  $f$, $\exists^\infty x \ f(x)
\leq g(x)$. We will  define a \computably\  random real $X$ and an
$A$-\computable\  martingale $L$ that succeeds on $X$ in the sense
of Schnorr, so that $A$ is not Low(\CR, \SR). In the following, $\alpha,
\beta, \gamma$ denote finite subsets of $\NN$, and $n_{\alpha}=
\sum_{i \in \alpha} 2^i$ (here $n_\emptyset = 0$).

 Let $\{M_e\}_{e\in\omega}$ be an effective listing of all partial \recursive\  martingales
with range included in $[1/2, \infty)$. At stage $t$, we have a
finite portion $M_e[t]$, whose domain is a subset of some set of the
form $2^{\le n}$ for some $n$. If $X$ is not \computably\ random,
then $\lim_{n\to\infty}M_e(X\restrict n) =\infty$ for some total $M_e$
by \cite{Schnorr:71}. Let
$$TMG = \{e: M_e\ {\rm is\ total}\}.$$  
For finite sets $\alpha$, $\beta$, let us in this proof say that
$\alpha$ is a \emph{strong subset} of $\beta$ (denoted
$\alpha\subseteq^+\beta$) if $\alpha\subseteq\beta$ and moreover for
each $i\in\omega$, if $i\in\beta-\alpha$, then $i>\max(\alpha)$. Thus 
the possibility that $\beta$ contains an element smaller than some
element of $\alpha$ is ruled out.

For certain $\alpha$, and all  those included in $TMG$, we will
define strings $x_\alpha$ in such a way that $\alpha \subseteq^+
\beta \Rightarrow x_\alpha \preceq x_\beta$. We choose the strings
in such a way that $M_e(x_\alpha)$ is bounded by a fixed constant
(depending on $e$) for each total $M_e$ and each $\alpha$
containing $e$. Then the set of integers
$$X= \bigcup_{e\in\omega} x_{TMG\inter [0,e]}$$
\noindent is a \computably\  random real. On the other hand, we are
able to define an $A$-\computable\  martingale $L$ which Schnorr
succeeds on $X$. We give an inductive definition of the strings
$x_\alpha$, ``scaling factors'' $p_\alpha  \in \QQ^+$ (positive
rationals) (we do not define $p_\emptyset$), and partial \computable\
martingales $M_\alpha$ such that if $x_\alpha$ is defined, then
\begin{equation} \label{bb:steps}
M_\alpha(x_\alpha) \ \mbox{converges in} \ g(|x_\alpha|) \
\mbox{steps and} \ M_\alpha(x_\alpha) <2.
\end{equation}

\noindent It will be clear that $A$ can decide if $y=x_\alpha$, given
inputs $y$ and $\alpha$.

Let $x_\emptyset=\lambda$, and let $M_\emptyset$ be the constant
zero function. (We may assume that $g$ is such that $M_\emptyset(\lambda)$
converges in $g(0)$ steps.) Now suppose $\alpha = \beta \cup \{e\}$,
where $e
> \max(\beta)$, and inductively suppose that (\ref{bb:steps})
holds for $\beta$. Let
$$p_\alpha= {\displaystyle \frac{1}{2}} 2^{-|x_\beta|}(2-M_\beta(x_\beta)),$$
\noindent and let $ M_\alpha = M_\beta + p_\alpha M_e$. Since $M_e$
is a martingale on its domain, $M_e(z) \leq 2^{|z|}$ for any $z$. So,
writing $b=M_\beta(x_\beta)$, we have $M_\alpha(x_\beta)=b+p_\alpha
M_e(x_\beta)<b+p_\alpha 2^{|x_\beta|} = b+\frac{1}{2}(2-b) =
1+\frac{b}{2}< 1+\frac{2}{2}=2$ if $M_\alpha(x_\beta)$ is defined.

To define $x_\alpha$, we look for a sufficiently long $x\succeq
x_\beta$ such that $M_\alpha$ does not increase from $x_\beta$ to
$x$ and $M_\alpha(x)$ converges in $g(|x|)$ steps. In detail, for
larger and larger $m>|x_\beta|$, $m \geq 4n_\alpha$, if no string
$y$, $|y| <m$, has been designated to be $x_\alpha$ as yet, and if
$M_\alpha(z)$ (i.e., each $M_e(z), e \in \alpha$) converges in
$g(m)$ steps, for each string $z$ of length $\le m$, then choose
$x_\alpha$ of length $m$, $x_\beta \prec x_\alpha$ such that
$M_\alpha$ does not increase anywhere from $x_\beta$ to $x_\alpha$.

\begin{claim}
If $\alpha \subseteq TMG$, then $x_\alpha$
and $p_\alpha$ (the latter only if $\alpha\ne\emptyset$) are
defined.
\end{claim}
\unskip

\begin{proof} The claim is trivial for $\alpha =
\emptyset$. Suppose that it holds for  $\beta$, and $\alpha = \beta \cup
\{e\}\subseteq TMG$, where $e
> \max(\beta)$. Since the function
$$f(m)= \mu s \quad \forall e  \in \alpha, \forall x\qquad [\, |x|\leq m
\Rightarrow M_e(x)\ {\rm converges\ in\ {\it s}\ steps}\,]$$

\noindent is \computable, there is a least $m\geq 4n_\alpha$, $m >
|x_\beta|$ such that $g(m) \geq f(m)$. Since there is a path down
the tree starting at $x_\beta$, where $M_\alpha$ does not increase,
the choice of $x_\alpha$ can be made.\qquad
\end{proof}

\begin{claim}\label{endextension}
If $\beta\subseteq^+\alpha$ are finite sets, then $M_\beta(x)\le
M_\alpha(x)$ for all $x$.
\end{claim}
\unskip

\begin{proof}This is clear by induction from the case
$\alpha=\beta\cup\{e\}$, i.e., the case where $\alpha-\beta$ has
only one element.\qquad\end{proof}

\begin{claim}  $X$ is \computably\  random. \end{claim}

\begin{proof} Suppose that $M_e$ is total. Let $\alpha= TMG
\cap [0,e]$. Suppose $\alpha \subseteq \gamma$, $\gamma'= \gamma
\cup \{i\}$, $\max(\gamma)<i$, and $\gamma'\subseteq TMG$. Then
$\alpha\subseteq^+\gamma\subseteq^+\gamma'$. Hence by Claim
\ref{endextension}, for each $x$ with $x_\gamma \preceq  x \preceq
x_{\gamma'}$, we have
$$p_\alpha M_e(x) \leq M_\alpha(x)\le M_{\gamma'}(x) \leq  M_{\gamma'}(x_\gamma) <2,$$
\noindent and hence $M_e(x)< 2/p_\alpha$ for each $x \prec X$, and so
the capital of $M_e$ on $X$ is bounded\mbox{.\qquad}\end{proof}

\begin{claim} There is a martingale $L \le_T A$ which  effectively succeeds on $X$. In fact,
\end{claim}
$$\exists^\infty x \prec X  \ L(x) \ge \  \left\lfloor \frac{|x|}{4} \right\rfloor.$$
\unskip

\begin{proof} For a string $z$, let $r(z)= \lfloor |z|/2
\rfloor$. We let $L = \sum_\alpha L_\alpha$, where $L_\alpha$ is a
martingale with initial capital $L_\alpha(\lambda) =2^{-n_\alpha}$,
which bets everything along $x_\alpha$ from $x_\alpha \uhr
r(x_\alpha)$ on. More precisely, if $x_\alpha$ is undefined, then
$L_\alpha$ is constant with value $2^{-n_\alpha}$. Otherwise, for
convenience we let $x=x_\alpha\uhr 2r(x_\alpha)$ and work with $x$
instead of $x_\alpha$; we define $L_\alpha$ on a string $y$ as
follows:

\begin{itemize}

\item If $y$ does not contain ``half of $x$,'' i.e., if $x\uhr
r(x)\not\preceq y$, then just let $L_\alpha(y)=2^{-n_\alpha}$.

\item If $y$ does contain ``half of $x$'' but $y$ and $x$ are
incompatible, then let $L_\alpha(y)=0$.

\item If $y$ contains ``half of $x$'' and $x$ and $y$ are
compatible, then let $L_\alpha(y)= 2^{-n_\alpha}2^{\min(|y|-r(x),
r(x))}$.

\end{itemize}

Thus if $y$ contains $x$, then $L_\alpha(y)=2^{r(x)-n_\alpha}$, so we
make no more bets once we extend $x_\alpha$, and if $x$ contains $y$,
then $L_\alpha(y)=2^{|y|-r(x)-n_\alpha}$; i.e., we double the capital
for each correct bit of $x$ beyond $x\uhr r(x)$.

Note that $L(\lambda)=\sum_\alpha 2^{-n_\alpha}$, and, as each
$k\in\omega$ has a unique binary expansion and hence is equal to
$n_\alpha$ for a unique finite set $\alpha$, we have
$L(\lambda)=\sum_{k\in\omega}2^{-k}=2$. Moreover, it is clear that
each $L_\alpha$ satisfies the martingale property
$L_\alpha(x0)+L_\alpha(x1)=2L_\alpha(x)$, and hence so does $L$.

$L$ effectively succeeds on $X$. Indeed, as $|x_\alpha|\ge
4n_\alpha$, we have $L_\alpha(x_\alpha)= 2^{r(x_\alpha)-n_\alpha}\ge
2^{\lfloor |x_\alpha|/2\rfloor-\lfloor |x_\alpha|/4|\rfloor}\ge
2^{\lfloor |x_\alpha|/4\rfloor}\ge \lfloor |x_\alpha|/4\rfloor$
since $2^q\ge q$ for each $q\in\omega$.

Finally, we show that $L\leq_T A$. Given input $y$,
we use $g$ to see if some string $x$, $|x| \leq 2|y|$, is $x_\alpha$.
If not, $L_\alpha(y) = 2^{-n_\alpha}$. Else we determine
$L_\alpha(y)$ from $x$ using the definition of $L_\alpha$.\qquad\end{proof}

The second ingredient to the proof of Theorem \ref{allonlow}(II) is
the following fact of independent interest.
\begin{proposition}\label{frank} If
$A$ is hyperimmune-free and \ce\ traceable, then $A$ is \computably\
traceable.
\end{proposition}
\unskip

\begin{proof} Let $f\le_T A$, and let $h$ be as in the
definition of \ce\ traceability. Let $g(x)=\mu s(f(x)\in W_{h(x),s})$
(where $W_{e,s}$ is the approximation at stage $s$ to the \ce\ set
$W_e$). Then $g\le_T A$, and so since $A$ is hyperimmune-free, $g$ is
dominated by a \recursive\  function $r$. Thus if we replace
$W_{h(x)}$ by $W_{h(x),r(x)}$, we obtain a \recursive\  trace for
$f$.\qquad
\end{proof}

Lemma \ref{andre} and Proposition \ref{frank} together establish
Theorem \ref{allonlow}(II): if $A$ is $\LLL(\CR, \SR)$, then $A$ is
\ce\ traceable by Theorem \ref{allonlow}(I), and hyperimmune-free by
Lemma \ref{andre}. Thus by Proposition \ref{frank}, $A$ is
\computably\ traceable.

As a corollary, we obtain an answer to the question of Ambos-Spies
and Ku\v{c}era.

\begin{corollary} \label{main}
A real $A$ is $S$-low iff it is $S_0$-low.
\end{corollary}
\unskip

\begin{proof} This follows by Theorem \ref{TZbigtheorem}
and Theorem \ref{allonlow}(II), since each \computably\  traceable
real is  $S_0$-low.\qquad \end{proof}

\section{Lowness notions related to Chaitin's halting probability}

Recall that $A$ is array \recursive\  if there is a function $f
\le_{wtt} \ES'$ bounding all functions \computablefrom\ $A$ on
almost all inputs.

\begin{theorem} If $\Omega \in$  \SR$^A$, then $A$ is array \recursive.  
\end{theorem}
\unskip

\begin{proof} We show that the function $\beta(x)= \mu s \
\Omega_s \uhr 3x = \Omega \uhr 3x$ dominates each function  $\alpha
\le_T A$. Since $\beta \le_{wtt} \Omega\le_{wtt}0'$, this shows that
$A$ is array \recursive.

Given $\alpha \le_T A$, consider the $A$-\recursive\  martingale
$M=\sum_p M_p$, where $M_p$ is the martingale which has the value
$2^{-p} $ on all strings of length up to $p$ and then doubles the
capital along the string $y=\Omega_{\alpha(p)} \uhr 3p$, so that
$M_p(y) =2^p$. Note that $M(z) $ is rational for each $z$. If
$\alpha(p) > \beta(p)$ for infinitely many $p$, then $M$ Schnorr
succeeds on $\Omega$, a contradiction.\qquad \end{proof}

\begin{corollary}  If $A$ is \ce, then $\Omega \in$  \SR$^A$ if{}f $A$ is \ce\ traceable. 
\end{corollary}
\unskip

\begin{proof} For an \ce\ set $A$, array \recursive\  implies
\ce\ traceable by the work of Ishmukhametov \cite{Ishmukhametov:99}.\qquad
\end{proof}

In \cite{KMS} it is shown that \ce\ traceable degrees do not contain
diagonally non\recursive\  functions, and hence, by a result of
Ku\v{c}era \cite{Kucera:84}, the \ce\ traceable degrees have measure
zero. On the other hand, every real $A$ which is Martin-L\"of random
relative to $\Omega$ satisfies that $\Omega$ is \MR$^A$, by van
Lambalgen's theorem \cite{van.Lambalgen:90}, and hence the measure
of the set of $A$ such that $\Omega$ is \SR$^A$ is one; thus $A$ \ce\
traceable is not equivalent to $\Omega\in$\SR$^A$. Also,
$\Omega\in$\SR$^A$ is not equivalent to $A$ being array \recursive,
as we now show.

The following notion of forcing appears implicitly in
\cite{Downey.Jockusch.ea:96}.

\begin{definition}
A tree $T$ is a set of strings $\sigma\in 2^{<\omega}$ such that if
$\sigma\in T$ and $\tau$ is a substring of $\sigma$, then $\tau\in
T$. A tree $T$ is full on a set $F\subseteq\omega$ if whenever
$\sigma\in T$ and $|\sigma|\in F$, then $\sigma0\in T$ and
$\sigma1\in T$. Let $F_n$, $n\in\omega$, be finite sets such that
each $F_n$ is an interval of $\omega$, $|F_{n+1}|>|F_n|$, and
$\Union_n F_n=\omega$. The sequence $F_n$, $n\in\omega$, is called a
\emph{very strong array}. Let $P$ be the set of \recursive\  perfect
trees $T$ such that $T$ is full on $F_n$ for infinitely many $n$.
Order $P$ by $T_1\le_P T_2$ if $T_1\subseteq T_2$. The partial order
$(P,\le_P)$ is a notion of forcing that we call very strong array
forcing.
\end{definition}

\begin{theorem}
For each real $X$ there is a hyperimmune-free real $A$ such that no
real \computablefrom\ $X$ is in $\SR^A$. In particular, as
hyperimmune-free implies array \recursive, there is an array
\recursive\  real $A$ such that $\Omega\not\in\SR^A$.
\end{theorem}
\unskip

\begin{proof} Let $A$ be sufficiently generic for very
strong array forcing. Then $A$ is hyperimmune-free, as may be proved
by modifying the standard construction of a hyperimmune-free degree
\cite{Martin.Miller:68} to work with trees that are full on
infinitely many $F_n$, $n\in\omega$.

Moreover, for each real $B$ \computablefrom\ $X$, there is an $n$
(hence infinitely many $n$) such that $A$ agrees with $B$ on $F_n$.
Indeed, given a condition $T$, a condition extending $T$ and ensuring
the existence of such an $n$ is obtained as a full subtree of $T$.

Hence no real $B$ \computablefrom\ $X$ is Schnorr random relative to
$A$. Indeed the measure of the set of those oracles $B$ that agree
with $A$ on infinitely many $F_n$ is zero, and it is easy to see
that the measure of those $B$ such that, for some $k>n$, $A$ and $B$
agree on $F_k$, goes to zero effectively as $n\to\infty$. Hence
there is a Schnorr test relative to $A$ which is failed by any such
$B$, as desired.\qquad\end{proof}

\begin{question} Characterize the (\ce)~sets of integers $A$ such
that $\Omega$ is computably random relative to $A$. Does this depend
on the version of $\Omega$ used?
\end{question}

\end{document}